\documentclass{amsart}
\usepackage{amsmath,amssymb,amsthm, tikz}
\usepackage{xcolor}
\usepackage{fancyhdr}

\newtheorem{theorem}{Theorem}
\newtheorem{lemma}[theorem]{Lemma}

\newtheorem{proposition}[theorem]{Proposition}

\theoremstyle{definition}
\newtheorem{definition}[theorem]{Definition}
\newtheorem{example}[theorem]{Example}

\newtheorem{remark}[theorem]{Remark}
\newtheorem*{Notation warning}{Warning on notations}
\newtheorem*{Overview}{Overview}
\newtheorem*{Sketch of proof}{Sketch of proof for Theorem C}

\newtheorem{thmintro}{Theorem}

\newtheorem{corintro}[thmintro]{Corollary}

\newcommand{\gs}{\sigma }

\newcommand{\bdry}{\partial_\infty}
\newcommand{\h}{{\mathfrak h}}
\newcommand{\hh}{\hat{\mathfrak h}}

\newcommand{\hk}{\hat{\mathfrak d}}
\newcommand{\A}{{\mathfrak a}}
\newcommand{\hA}{\hat{\mathfrak a}}

\newcommand{\B}{{\mathfrak b}}
\newcommand{\hB}{\hat{\mathfrak b}}

\newcommand{\C}{{\mathfrak c}}
\newcommand{\hC}{\hat{\mathfrak c}}

\DeclareMathOperator{\cat}{CAT(0)}
\DeclareMathOperator{\Gir}{Girth}
\DeclareMathOperator{\Cay}{Cay}
\DeclareMathOperator{\Isom}{Isom}
\DeclareMathOperator{\R}{\bf{R}}

\def\aut{\mathop{\mathrm{Aut}}\nolimits}

\usepackage{hyperref}

\usepackage[
backend=biber,
style=alphabetic,
url=false,
doi=false,
isbn=false
]{biblatex}

\usepackage{cleveref}

\thanks{}

\date{\today}

\begin{document}
\title[On the Girth of Groups Acting on CAT(0) Cube Complexes]{On the Girth of Groups Acting on CAT(0) Cube Complexes}
\author[A. banerjee]{Arka banerjee}
\address{Auburn University, 221 Parker Hall, Auburn, AL 36849, USA}
\email{banerjee20arka@gmail.com}
\author[D. Gulbrandsen]{Daniel Gulbrandsen}
\address{Adams State University, 208 Edgemont Blvd, Alamosa, CO 81101, USA}
\email{dangulbrandsen@yahoo.com}
\author[P. Mishra]{Pratyush Mishra}
\address{Wake Forest University, 127 Manchester Hall, Winston-Salem, NC 27109, USA}
\email{mishrapratyushkumar@gmail.com}
\author[P. Parija]{Prayagdeep Parija}
\address{Virginia Tech, Mcbryde Hall, 225 Stranger Street, Blacksburg, VA 24060, USA }
\email{prayagdeep@gmail.com}

\begin{abstract}
 We obtain a sufficient condition for lattices in the automorphism group of a finite dimensional CAT(0) cube complex to have infinite girth. As a corollary, we get a version of Girth Alternative for groups acting geometrically: any such group is either \{locally finite\}-by-\{virtually abelian\} or it has infinite girth. We produce counterexamples to show that the alternative fails in the general class of groups acting cocompactly on finite dimensional CAT(0) cube complexes by obtaining examples of non-virtually solvable groups which satisfy a law.
\end{abstract}

\maketitle
\begin{sloppypar}

\tableofcontents

\section{Introduction}

The \emph{girth} of a finite graph is defined as the length of its shortest non-trivial cycle.
If the graph has no cycle then its girth is defined to be infinite. 
Given a finitely generated group $\Gamma$ with a fixed generating set $S$, it makes sense to talk about the girth of it's Cayley graph, denoted $\Gir(\Cay(\Gamma,S))$. Then, one can define the \textbf{girth} of the group $\Gamma$ as the supremum of the girth of the Cayley graphs over all finite  generating sets $S$ of $\Gamma,$ i.e. 
$$\Gir(\Gamma)=\sup_{\Gamma=\langle S \rangle, |S|<\infty} \bigg \{\Gir(\Cay(\Gamma,S))\bigg \}$$
Note that $\Gir(\Cay(\Gamma,S))$ measures the length of the shortest non-trivial word relation in the Cayley graph $\Cay(\Gamma, S)$. Some examples of groups with finite girth are finite groups, non-cyclic nilpotent groups and solvable groups. In fact, any finitely generated group satisfying a non-trivial law has finite girth, \cite{Schleimer02}, \cite{AKHMEDOV2003198}. 
Moreover, if a group has a finite index subgroup with finite girth, then the group itself has finite girth [see \cite{Schleimer02} for details]. 
In particular, virtually solvable groups have finite girth.
On the other hand, any free group has infinite girth, in particular $\mathbb{Z}$ has infinite girth (which can be seen as a degeneracy case).

\medskip

Given a class of groups, it is natural to ask if one can classify each group in the class in terms of finite and infinite girth. Motivated by the celebrated Tits Alternative \cite{Tits1972FreeSI}, Girth Alternative \cite{akhmedov2005girth} is defined as follows.
\medskip

\noindent
\textbf{Girth Alternative:} \label{thm: GA}\textit{For a given class $\mathcal{C}$ of finitely generated groups, $\mathcal{C}$ is said to satisfy the Girth Alternative if for any group $\Gamma\in \mathcal{C}$, either $\Gamma$ is virtually solvable (and hence has finite girth unless isomorphic to $\mathbb{Z}$) or girth of $\Gamma$ is infinite.}
\medskip

The Girth Alternative is similar in spirit to the Tits Alternative (which was first introduced by Jacques Tits in 1972, where he proved the alternative for the class of linear groups). For a general class of groups, the Tits Alternative can be stated as follows.
\medskip

\noindent
\textbf{Tits Alternative:} \label{thm:TA}\textit{For a given class $\mathcal{C}$ of finitely generated groups, $\mathcal{C}$ is said to satisfy Tits Alternative if for any group $\Gamma \in \mathcal{C}$, either $\Gamma$ is virtually solvable or $\Gamma$ contains a non-abelian free subgroup $\mathbb{F}_2$.}
\medskip

Some classes of finitely generated groups satisfying the Tits Alternative are linear groups (\cite{Tits1972FreeSI}), Gromov hyperbolic groups (follows from the fact that it contains a cyclic subgroup of finite index or $\mathbb{F}_2$ as a subgroup, see Theorem 5.3.E in \cite{Gro87}), mapping class groups of compact surfaces (\cite{MR0745513}, \cite{10.1215/S0012-7094-83-05046-9}), Out($\mathbb{F}_n$) (\cite{Bestvina1997TheTA}, \cite{Bestvina1997TheTA2}), groups acting on CAT(0) cube complex (\cite {sageev_wise_2005}, \cite{Caprace_2011}), etc. On the other hand, Thompson's group F does not satisfy Tits Alternative: neither Thompson's group F contains a copy of $\mathbb{F}_2$ (\cite{MR0782231}) nor it is virtually solvable (\cite{MR2707536}).
 
\medskip

It has been observed that for certain classes of groups for which Tits Alternative holds, the property of having infinite girth coincides with the property of containing a non-abelian free group \footnote{Although, this observation does not seem to hold in general, see \cite{akhmedov2014girth}}.
A comprehensive theory of the Girth Alternative for different classes of groups has been developed by many authors in the past two decades. 
In \cite{AKHMEDOV2003198}, \cite{akhmedov2005girth}, Akhmedov has introduced and proved the Girth Alternative for the class of hyperbolic, linear,
one-relator and $\mbox{PL}_+(I)$ groups. In \cite{yamagata}, Yamagata proves the Girth Alternative for convergence groups and irreducible subgroups of the mapping class groups. Independently in \cite{nakamura2011girth},
Nakamura proves the Alternative for all subgroups of mapping class groups and also for the
subgroups of Out($\mathbb{F}_n$) containing the irreducible elements having irreducible powers. Recently, Akhmedov and the third author of this article have studied the Girth Alternative phenomenon for HNN extensions and amalgamated free products \cite{AkhmedovMishra2023}.
\medskip

Therefore, it is natural to study the Girth Alternative for the class of groups for which Tits Alternative has been investigated. 
In this paper, we focus on groups acting on $\cat$ cube complex. 
Different versions of Tits Alternative for groups acting on CAT(0) cube complex are known from the work of  Sageev--Wise \cite{sageev_wise_2005} and  Caprace--Sageev \cite{Caprace_2011}. Inspired by the work of Caprace--Sageev \cite{Caprace_2011}, we prove a Girth Alternative for lattices in the group of automorphisms of finite dimensional cube complex.

\begin{thmintro}\label{intro:v2}
    \emph{Let $X = X_1 \times \dots \times X_n$ be a product of $n$ irreducible, unbounded, and locally compact $\cat$ cube complexes such that $\aut(X_i)$ acts cocompactly and essentially on $X_i$ for all $i$.}       
      
      \emph{Then for any (possibly non-uniform) lattice $\Gamma \leq \aut(X)$ acting properly  on $X$, 
either $\Gamma$ is \{locally finite\}-by-\{virtually abelian\} or 
$\Gir(\Gamma)=\infty$.}
\end{thmintro}

Note that if a group has an upper bound on the order of the finite subgroups then any locally finite subgroup of that group has finite order. Therefore, the following is an immediate consequence of Theorem~\ref{intro:v2}.
\begin{corintro}\label{intro:v3}  
\emph{Let $X = X_1 \times \dots \times X_n$ be a product of $n$ irreducible, unbounded, and locally compact $\cat$ cube complexes such that $\aut(X_i)$ acts cocompactly and essentially on $X_i$ for all $i$.       
         Suppose $\Gamma \leq \aut(X)$ is a lattice (possibly non-uniform) acting properly on $X$ and $\Gamma$ has a bound on the order of its  finite subgroups.}

\emph{Then  $\Gamma$ is either virtually abelian or 
$\Gir(\Gamma)=\infty$. In particular, $\Gamma$ satisfies the Girth Alternative.}
\end{corintro}
We obtain Theorem~\ref{intro:v2} as a consequence of the (proof of) following more general theorem that gives  sufficient conditions for a lattice to have infinite girth.

\begin{thmintro}\label{intro:v1}
  \emph{Let $X = X_1 \times \dots \times X_n$ be a product of $n$ irreducible, unbounded, locally compact $\cat$ cube complexes such that $\aut(X_i)$ acts cocompactly and essentially on $X_i$ for all $i$.}

\emph{Then for any (possibly non-uniform) lattice $\Gamma \leq \aut(X)$, 
either $\Gamma$ virtually fixes some point in $X\cup \bdry X$ or 
$\Gir(\Gamma)=\infty$.}
\end{thmintro}

In contrast to the above theorems, it is also interesting to investigate groups which do not satisfy Girth Alternative. This essentially boils down to finding groups satisfying a law, but which are not virtually solvable. 
In \cite{Cornulier_Mann}, the authors produce such examples of groups. More precisely, they construct finitely generated residually p-groups which satisfy a law, but are not virtually solvable. In fact, their examples show that both the Girth Alternative and Tits Alternative fail for the class of residually finite groups. However, we do not know if the examples of Cornulier and Mann can be realized as groups acting on CAT(0) cube complex. Therefore, we provide our own counterexample in Section \ref{section:counterexample} to show that Girth Alternative, as stated in Corollary~\ref{intro:v3}, fails for the class of groups acting cocompactly on finite dimensional CAT(0) cube complexes.

\medskip

\begin{Sketch of proof} 
There are three main ingredients used in our proof. The first is a theorem of Nakamura \cite{nakamura2011girth} (see Theorem~\ref{t:criteria for infinite girth}) that gives a sufficient condition for a group to have infinite girth.
In particular, the condition asks for two elements that generate a free subgroup. 
In order to ensure such elements always exist under the hypotheses of Theorem~\ref{intro:v1}, we use our second ingredient, which is a theorem of Caprace and Sageev \cite{Caprace_2011} (see Theorem~\ref{Regular elements in lattices}).
However, infinite girth does not follow just from a non-abelian free subgroup. 
Rather, it follows from more subtle properties of generators of the free subgroup, which are required by Nakamura's theorem.
To account for these extra properties, we have to choose `correct' generators.
To this end we use our third ingredient, Lemma~\ref{l:many facing hyperplanes}, which asserts the existence of lots of mutually disjoint halfspaces in the cube complex under certain conditions.
Combining ingredients two and three yields infinitely many candidates for generators of different free subgroups.
 The primary objective in the proof of Theorem~\ref{intro:v1} is then to find the `correct' generators from this lot that satisfy the properties needed to apply Nakamura's theorem.
As a result, we obtain infinite girth.
\end{Sketch of proof}

\begin{Overview}The paper is organized as follows. In Section~\ref{section:preliminaries}, we discuss some basic concepts pertaining to the theory of CAT(0) cube complexes. Readers who are familiar with CAT(0) cube complexes may skip this section.  
We prove Theorem~\ref{intro:v1} and Theorem~\ref{intro:v2} in Section~\ref{section:GAforlattices} where they appear as Theorem~\ref{premain theorem} and Theorem~\ref{Girth alternative v2}, respectively.
In Section~\ref{section:counterexample}, we give an example which shows Corollary~\ref{intro:v3} does not hold if both the assumptions properness and having a bound on the order of finite subgroups are omitted.
\end{Overview}

\section{Preliminaries}\label{section:preliminaries} 

\subsection{Cube Complexes} 

Let $I=[0,1]$. We define an {\bf $n$-cube}, or simply {\bf cube}, as the product $I^n$ (with $I^0=\{0\}$). A {\bf face} of $I^n$ is the restriction of some number (possibly zero) of its coordinates to $0$ or $1$, and a {\bf midcube} of $I^n$ is the restriction of precisely one of its coordinates to $1/2$. Cleary, midcubes of $I^n$ can be realized as $(n-1)$-cubes and we will often treat them as such. We will refer to $0$-cubes, $1$-cubes, and $2$-cubes as vertices, edges, and squares respectively.
\medskip

A {\bf cube complex} is a cell complex where the cells are $n$-cubes, for various $n$, and the attaching maps are Euclidean isometries in faces. We say the dimension of a cube complex $X$ is $n$, or $X$ is {\bf finite dimensional}, if $X$ contains $n$-cubes but not $(n+1)$-cubes, for some $n$. If no such $n$ exists, we say that $X$ is {\bf infinite dimensional}. $X$ is {\bf locally finite} if every vertex of $X$ meets only finitely many cubes. In the setting of cube complexes, being locally finite is equivalent to being locally compact.
\medskip

Two cube complexes $X$ and $Y$ are {\bf isomorphic} if there is a bijection $\phi$ between their vertex sets such that, if $c$ is an $n$-cube in $X$ then $\phi(c)$ is an $n$-cube in $Y$. In particular, an isomorphism of cube complexes maps adjacent vertices to adjacent vertices. 
\medskip

A {\bf subcomplex} $Y$ of a cube complex $X$ is any collection of cubes in $X$ inheriting the same attaching maps. We will denote subcomplexes by $Y\le X$. We call $X$ {\bf irreducible} if there do not exist cube complexes $X_1,X_2$ with $X=X_1\times X_2$. Note that, if $X=X_1\times X_2$, then $X$ contains subcomplexes isomorphic to $X_1$ and $X_2$. Cube complexes need not be connected.
\medskip

A cube complex $X$ is a CAT(0) space if extending the Euclidean metric on individual cubes to the path length metric on $X$ yields a CAT(0) space, in the traditional sense. That is, triangles are at least as thin as their comparison triangles in Euclidean space. Due to Gromov \cite{Gro87}, there is a combinatorial requirement that, if satisfied by $X$, will guarantee that $X$ is locally CAT(0). A finite dimensional cube complex $X$ is {\bf locally CAT(0)} if the link of every vertex in $X$ is a simplicial flag complex. If, in addition, $X$ is simply-connected then being locally CAT(0) implies that $X$ is CAT(0). The equivalence between the being locally CAT(0) and (locally) satisfying the classical definition of CAT(0) was extended to infinite dimensional CAT(0) cube complexes in \cite{https://doi.org/10.1112/jtopol/jts035}.
\medskip

From now on, unless stated otherwise, all cube complexes considered are assumed to be CAT(0), locally finite, and finite dimensional. Note that for all such cube complexes $X$, the decomposition $X=X_1\times\cdots\times X_n$ into irreducible factors is unique up to a permutation of its factors by the following proposition.
\medskip
\begin{proposition}[\cite{Caprace_2011}]\label{p:product decomposition}
    A finite dimensional $\cat$ cube complex $X$ admits a canonical decomposition
    \[
    X=X_1\times X_2\times \ldots \times X_p
    \]
    into a product of irreducible cube complexes $X_i$. Every automorphism of $X$ preserves that decomposition, up to a permutation of possibly isomorphic factors. In particular, the image of canonical embedding
    \[
    \aut(X_1)\times \aut(X_2)\times \ldots \times \aut(X_p)\hookrightarrow \aut(X)
    \]
    has finite index in $\aut(X)$.
\end{proposition}
When it is necessary to treat $X$ as a metric space, we will use $d$ to denote the CAT(0) metric. With this metric, $X$ is a unique geodesic space.

\subsection{Hyperplanes and Halfspaces}

For a cube complex $X$, we say two edges $e_1,e_2\in X$ are {\bf related} if $e_1\cap e_2=\emptyset$ and there is a square in $X$ containing $e_1$ and $e_2$.

Informally, $e_1$ and $e_2$ are related if they are opposite edges of a square. A {\bf hyperplane} in $X$, typically denoted as $\hh$, is the transitive closure of this relation on a single edge. The collection of all hyperplanes in a cube complex $X$ is denoted $\hat{\mathcal{H}}_X$. For the {\bf geometric realization of a hyperplane}, first observe that any nonempty collection of related edges in an $n$-cube $c$ uniquely determine a midcube of $c$. Midcubes in $X$ therefore inherit a relation from the edges. One then sees that two midcubes $M$ and $N$ are related whenever their intersection is a common face. The geometric realization of $\hh$ is the transitive closure of this relation on a single midcube. In this instance, midcubes $M$ and $N$ inherit attaching maps from their corresponding cubes.
\medskip

For the remainder, we use the same symbol, $\hh$, to denote both a hyperplane and its geometric realization.
\medskip

Two hyperplanes $\hh_1$ and $\hh_2$ are {\bf transverse}, denoted $\hh_1\pitchfork\hh_2$, if there is a square in $X$ containing edges $e_1$ and $e_2$ with $e_1\in\hh_1,\;e_2\in\hh_2$, and $e_1\cap e_2\neq\emptyset$. Note that $\hh_1\pitchfork\hh_2$ if and only if their geometric realizations intersect. If $\hh_1$ and $\hh_2$ are not transverse, they are {\bf nested}. We say $\hh_1$ and $\hh_2$ are {\bf separated} if there is a third hyperplane containing $\hh_1$ and $\hh_2$ in complimentary halfspaces, and they are {\bf strongly separated} if they are separated and there does not exist a hyperplane transverse to both $\hh_1$ and $\hh_2$.
\medskip

\noindent{\bf Properties of Hyperplanes}. Let $X$ be a CAT(0) cube complex and $\hh\in\hat{\mathcal{H}}_X$
\begin{enumerate}
    \item $\hh$ is a CAT(0) cube complex in its own right.
 \item $\hh$ separates $X$ into two components.    
\end{enumerate}
\medskip

\noindent Recent proofs of these properties can be found in \cite{10.1007/978-3-319-43674-6_4}.
We denote the closures of the two components of $X\setminus\hh$ by $\h$ and $\h^*$, which we also refer to as {\bf halfspaces}.
Choosing which component to call $\h$ for all $\hh\in\hat{\mathcal{H}}_X$ defines a {\bf labeling} of $\hat{\mathcal{H}}_X$. The collection of halfspaces of $X$ is denoted $\mathcal{H}_X$. A halfspace $\h$ is {\bf deep} if the distance from $\hh$ to points in $\h$ is unbounded. Otherwise, $\h$ is {\bf shallow}. Halfspaces $\h_1$ and $\h_2$ are {\bf nested} if either $\h_1\subset\h_2$ or $\h_2\subset\h_1$. Note that if hyperplanes $\hh_1$ and $\hh_2$ are nested then they each have a halfspace which together form a nested pair. We also call $\h_1$ and $\h_2$ {\bf strongly separated} if $\hh_1$ and $\hh_2$ are strongly separated.

\subsection{Boundaries and Regular Points}

The collection of halfspaces of a CAT(0) cube complex $X$ forms a poset with partial order given by set-inclusion, and comes naturally equipped a complimentary map given by the involution $\h\to\h^*$. Let $\mathcal{H}_X^\circ$ denote the collection of ultrafilters of $\mathcal{H}_X$, that is, $\alpha\in\mathcal{H}_X^\circ$ satisfies the two conditions
\begin{enumerate}
    \item (Choice) for every pair $\h,\h^*\in\mathcal{H}_X$, exactly one of them is in $\alpha$, and

    \item (Consistency) if $\h_1\in\alpha$ and $\h_2\in\mathcal{H}_X$ with $\h_1\subset\h_2$, then $\h_2\in\alpha$.
\end{enumerate}
Given an ultrafilter $\alpha\in\mathcal{H}_X^\circ$, any subset $\{\h_i\}\subseteq\alpha$ is guaranteed to satisfy the so called {\bf partial choice condition}: for every pair $\h,\h^*$, at most one of them is in $\alpha$.
\medskip

A nested sequence of halfspaces $\{\h_i\}$ is called {\bf descending} (or a {\bf descending chain}) if $\h_{i+1}\subseteq\h_{i}$ for all $i$. A descending sequence of halfspaces is {\bf terminating} if $\h_j=\h_N$, for some $N$ and all $j>N$. Otherwise $\{\h_i\}$ is {\bf nonterminating}. Given $\sigma\in\mathcal{H}_X^\circ$, we say $\sigma$ satisfies the {\bf descending chain condition (DCC)} if every descending sequence in $\sigma$ is terminating. Note that descending sequences necessarily satisfy the partial choice condition.
\medskip

One has a bijective correspondence between the vertex set of $X$ and ultrafilters in $\mathcal{H}_X^\circ$ satisfying the DCC given by $\theta:X^0\to\mathcal{H}^\circ$, where $\theta(v)=\{\h\in\mathcal{H}_X\;|\;v\in\h\}$. Given a vertex $v\in X$, choosing the halfspace containing $v$ to call $\h$ for all hyperplanes $\hh$ gives a labeling of $\hat{\mathcal{H}}_X$, which we refer to as being determined by $v$.
\medskip

Now, let $\bf{2}:=\{0,1\}$. Then $\bf{2}^{\hat{\mathcal{H}}}$ is compact in the Tychonoff topology and the choice condition induces an inclusion $\mathcal{H}_X^\circ\hookrightarrow{\bf2}^{\hat{\mathcal{H}}}$ in which $\mathcal{H}_X^\circ$ embeds as a closed subset which is therefore compact. We call the image of $\mathcal{H}_X^\circ$ the {\bf Roller compactification} and denote it by $\bar{X}$. To realize the image of $X^0$ in $\bf{2}^{\hat{\mathcal{H}}}$ concretely, start by defining $\phi_{\hh}:X^0\to\bf{2}$ by
\[
    \phi_{\hh}(v)=\left\{\begin{array}{cc}
        1 & {\rm if}\;x\in\h \\
        0 & \;{\rm if}\;x\in\h^*
    \end{array}\right.
\]
Then, the map with coordinate functions $\{\phi_{\hh}\}_{\hh\in\hat{\mathcal{H}}}$ gives an embedding $X^0\hookrightarrow\bar{X}$ as a dense open subset. See \cite{Roller2016PocSM} and
\cite{NV_PoissonBdry} for details. Note that the Roller compactification behaves well with respect to products. Indeed, if $X=X_1\times\cdots\times X_n$ then $\bar{X}=\bar{X}_1\times\cdots\times\bar{X}_n$. 
\medskip
The {\bf Roller Boundary} of $X$, denoted $\partial_RX$, is the remainder space $\bar{X}\setminus X$. It is a fact that $\partial_RX$ agrees with the collection of ultrafilters in $\mathcal{H}_X^\circ$ that do not satisfy the DCC. Note that an element $\xi\in\partial_RX$ necessarily contains a nonterminating descending sequence. Of paramount interest to us is the case that $\xi$ contains a nonterminating descending sequence of strongly separated halfspaces, $\{\h_i\}$. We call such $\xi$ a {\bf regular point}. These were originally defined in \cite{FERNÓS_2018} (and independently in \cite{kar2016ping}) as elements $\xi\in\partial_RX$ such that if $\h_1,\h_2\in\xi$ then there exists $\mathfrak{k}\in\mathcal{H}_X$ with $\mathfrak{k}\subset\h_1\cap\h_2$ and $\mathfrak{k}$ is strongly separated from both $\h_1$ and $\h_2$. They were also shown to be equivalent to the definition provided. \cite[Proposition 5.13]{Fernos2016Random} and Lemma \ref{l:intersecting hyperspaces} below imply that, if $\xi$ is a regular point containing a nonterminating descending sequence of strongly separated halfspaces, $\{\h_i\}$, and if $\xi'$ is a regular point with $\{\h_i\} \subseteq\xi'$ then $\xi'=\xi$. Due to this fact, we will often refer to a regular point, $\xi$, simply as a nonterminating descending sequence of strongly separated halfspaces, $\{\h_i\} \subseteq\xi$.
\medskip

The {\bf visual boundary} of a CAT(0) cube complex $X$ is the collection of equivalence classes of (CAT(0)) geodesic rays, where two such rays are equivalent if they remain within a bounded Hausdorff distance of each other. We denote the visual boundary by $\partial_\infty X$. 

\subsection{Groups Acting on Cube Complexes} 

An element $g\in \Isom(X)$ is an {\bf automorphism} if it acts as a cubical isomorphism of $X$. We let $\aut(X)$ denote the group of automorphisms of $X$. 
A subgroup of $\Gamma$ in $\aut(X)$ is said to act {\bf cocompactly} if there exists a compact subcomplex $K\le X$ whose translates cover $X$, i.e., $X=\Gamma\cdot K$. Under our assumption that $X$ is locally compact, this is equivalent to requiring that $X/\Gamma$ be compact. $\Gamma$ acts {\bf essentially} if given any halfspace $\h$ in $X$, we have $d(\Gamma\cdot\hh,\hh)$ is unbounded.
Note that, for $\Gamma$ to act essentially, the halfspaces of $X$ must be deep.

\medskip

For $X$ a CAT(0) cube complex, upon passing to its barycentric subdivision, elements $g\in \aut(X)$ come in two types (\cite{Haglund2007IsometriesOC}):
\begin{enumerate}
    \item (Elliptic) $g$ fixes a vertex in $X$;
    
    \item (Hyperbolic) $g$ is not elliptic and preserves a geodesic line.
\end{enumerate}
\medskip

\noindent A geodesic line preserved by $g$ is called an {\bf axis} for $g$. A {\bf rank one isometry} is a hyperbolic automorphism $g$ none of whose axes bound a flat halfplane, and a {\bf contracting isometry} is a rank one isometry with axis $l$ such that the diameter of the orthogonal projection to $l$ of any ball disjoint from $l$ is bounded above. 
An element in $\aut(X)$ is \textbf{regular} if it acts as a contracting isometry on each irreducible factor of $X$.
\medskip
 
Contracting isometries can be detected by looking at their action on strongly separated hyperplanes. An element $g\in \aut(X)$ is said to \textbf{double skewer} the pair of hyperplanes $\hh',\hh''$ if there is a nested pair of halfspaces $\h'\subset\h''$ of $X$ such that $\h''\subsetneq g\h'$.

\begin{lemma}[\cite{Caprace_2011}]\label{contracting isom and hyperplanes}
    Let $X$ be a finite dimensional $\cat$ cube complex and $g\in \aut(X)$ double skewers a pair  of strongly separated hyperplanes in $X$. Then $g$ is a contracting isometry.
\end{lemma}

Suppose $g\in \aut(X)$ double skewers $(\hh',\hh'')$.
Then for some choice of halfspace $\h'$ of the hyperplane $\hh'$,  $\{g^i\h'\}$ and $\{g^{-i}\h'^*\}$ give us two nonterminating descending sequence of halfspaces.
Moreover, if $\hh'$ and $\hh''$ are strongly separated hyperplanes, then both $\{g^i\h'\}$ and $\{g^{-i}\h'^*\}$ are nonterminating descending sequences of strongly separated halfspaces, and hence correspond to two regular points in $\partial_R(X)$.
In this case $g$ can be thought as acting along an axis whose end points are $\{g^i\h'\}$ and $\{g^{-i}\h'^*\}$.
We refer to $\{g^i\h'\}$ and $\{g^{-i}\h'^*\}$ as {\bf poles} of $g$.

\section{Girth Alternative for Lattices}\label{section:GAforlattices}

One crucial step towards proving Girth Alternative is to find a  sufficient condition for infinite girth.
Recall that free groups have infinite girth.
One classical way of producing a copy of free group on two generators as a subgroup in a given group $\Gamma$ is to use a version of the  \emph{ping-pong} lemma. 
This goes back to  Tits \cite{Tits1972FreeSI} where he proved the following.
\begin{proposition}[Free subgroup criterion]\label{free subgroup criterion}
    Let $\Gamma$ be a group acting on a set $X$. Suppose there exist elements $\sigma,\tau\in \Gamma$, subsets $U_\sigma,U_\tau\subset X$, and a point $x\in X$, such that 
    \begin{enumerate}
        \item $x\notin U_\sigma\cup U_\tau$
        \item $\sigma^k(\{x\}\cup U_\tau\}\subset U_\sigma$ for all $k\in \mathbb{Z}-\{0\}$
        \item $\tau^k(\{x\}\cup U_\sigma\}\subset U_\tau$ for all $k\in \mathbb{Z}-\{0\}$.    \end{enumerate}
        Then $\langle \sigma,\tau\rangle $ is non-abelian free subgroup of $\Gamma$.
\end{proposition}

However, mere containment of a copy of $\mathbb{F}_2$ does not guarantee infinite girth.
In fact, there exist groups with finite girth that have free non-abelian subgroups \cite{AKHMEDOV2003198}.
Akhmedov \cite{akhmedov2005girth} observed that similar ping-pong arguments can be applied to certain classes of groups for which Tits Alternative holds to prove that a given group has infinite girth. 
Generalizing and reformulating the work of \cite{akhmedov2005girth}, Nakamura obtained the following criterion for infinite girth in comparable generality to Proposition \ref{free subgroup criterion}.

\begin{theorem}[Criteria for infinite girth \cite{nakamura2011girth}, \cite{nakamura2008some}]\label{t:criteria for infinite girth}
    Let $\Gamma$ be a group acting on a set $X$ with a finite generating set $S=\{\gamma_1,\ldots,\gamma_n\}$. Suppose there exist elements $\sigma,\tau\in \Gamma$, subsets $U_\sigma,U_\tau\subset X$, and a point $x\in X$ such that
    \begin{align}
       \label{Nakamura 1}&\;\;\;\; x\notin(U_\sigma\cup U_\tau)\cup\bigcup_{\varepsilon=\pm1}\bigcup_{i=1}^n\gamma_i^\varepsilon(U_\sigma\cup U_\tau)
       \\
       \label{Nakamura 2}&\;\;\;\;\sigma^k\left(\{x\}\cup U_\tau\cup\bigcup_{\varepsilon=\pm1}\bigcup_{i=1}^n\gamma_i^\varepsilon(U_\tau) \right)\subset U_\sigma\;\;for\;all\;k\in\mathbb{Z}-\{0\},\;and
       \\
       \label{Nakamura 3}&\;\;\;\;\tau^k\left(\{x\}\cup U_\sigma\cup\bigcup_{\varepsilon=\pm1}\bigcup_{i=1}^n\gamma_i^\varepsilon(U_\sigma)\right)\subset U_\tau\;\;for\;all\;k\in\mathbb{Z}-\{0\}.
    \end{align}
    Then $\Gamma$ is a non-cyclic group with $\Gir(\Gamma)=\infty$.
\end{theorem}

Clearly, the elements $\sigma,\tau$ and subsets $U_\sigma,U_\tau$ in Theorem \ref{t:criteria for infinite girth} satisfy the conditions in Proposition \ref{free subgroup criterion} and therefore $\langle\sigma,\tau\rangle\le\Gamma$ must be a non-abelian free subgroup.
In general, properties \eqref{Nakamura 2} and \eqref{Nakamura 3} can be interpreted as analogous to the North-South dynamics exhibited by the action of a loxodromic element in a hyperbolic group on its boundary.
In the realm of irreducible $\cat$ cube complexes, the role of such isometries are played by contracting isometries. 
In general, a finite dimensional $\cat$ cube complex is a product of finitely many irreducible factors by Proposition \ref{p:product decomposition}.
Therefore, to get infinite girth of a group acting on a finite dimensional $\cat$ cube complex, we should look for isometries that act as contracting isometries when restricted to each irreducible factor. In other words, we are looking for regular elements.
\medskip

The following Theorem of \cite{Caprace_2011} says that regular elements always exist under certain conditions.

\begin{theorem}[Regular elements in lattices \cite{Caprace_2011}]\label{Regular elements in lattices}
    Let $X=X_1\times\ldots \times X_n$ be a product of irreducible, unbounded, locally compact
$\cat$ cube complexes such that $\aut(X_i)$ acts cocompactly and essentially on $X_i$ for each
$i$. Suppose that $\Gamma\le\aut(X)$ is a lattice (possible non-uniform).
Suppose that $\h_i\subset \mathfrak{k}_i$ are nested
halfspaces in each factor $X_i$. Then there exists a regular element $g\in \Gamma$ which simultaneously
double skewers these hyperplanes. That is to say, for each $i$, $g\mathfrak{k}_i \subsetneq \h_i$.  
\end{theorem}

The above Theorem reduces the problem of finding a regular element in $\Gamma\le\aut(X)$ to finding a pair of strongly separated hyperplanes in each irreducible factor of $X$.
For our purpose, we will need two such regular elements  whose axes in each irreducible factor avoid certain points at infinity.
In particular, we want their axes to not intersect at infinity in each irreducible factor.
 According to Theorem \ref{Regular elements in lattices}, if we can find a facing quadruple of hyperplanes in each irreducible factor $X_i$, each  containing two pairs of strongly separated hyperplanes, then the group elements that skewer the strongly separated pairs in each $X_i$ will have the desired property. Our next Lemma says that, under certain conditions, we always have such a collection of facing hyperplanes.

\begin{Notation warning} For the rest of the paper, the notation $\{a_i,b_i\}$ will mean the set consisting of two elements $a_i$ and  $b_i$ whereas the notation $\{a_i\}$ will stand for a sequence $\{a_i\}_{i=1}^\infty$. 
 We will sometimes write sets with braces and at other times without braces, and we reserve parentheses for ordered sets.
 \end{Notation warning}

\begin{lemma}[Abundance of facing hyperplanes]\label{l:many facing hyperplanes}
    Let $X$ be an irreducible finite dimensional $\cat$ cube complex where $\aut(X)$ acts essentially and without fixed points at infinity. Let $\{\hA,\hB,\hC\}$ be a facing triple in $X$ with $\hA$ and $\hB$ strongly separated. Then for any $n$, there exists a collection of $2n$ facing hyperplanes $\{\hA_1=\hA,\hB_1=\hB,\hA_2,\hB_2,\ldots,\hA_n,\hB_n\}$ such that each pair $\{\hA_i,\hB_i\}$ is strongly separated for all $i$.
\end{lemma}

An automorphism $g\in \aut(X)$ is said to \textbf{flip} a halfspace $\h$, if $g\cdot \h\subsetneq \h^*$. The main idea of the proof of the above lemma is to hit one pair of hyperplanes from the triple by an automorphism that flips a halfspace corresponding to the third hyperplane, and then keep repeating this process on the resulting hyperplanes. 
The following Theorem of \cite{Caprace_2011} says that any halfspace of $X$ can be flipped  under certain conditions.
\begin{theorem}[Flipping Lemma \cite{Caprace_2011}]\label{flipping lemma}
    Let $X$ be finite dimensional $\cat$ cube complex and $\Gamma\le\aut(X)$ be a group acting essentially without a fixed point at infinity. Then for any halfspace $\h$, there is some $\gamma\in \Gamma$ such that $\gamma\cdot \h\subsetneq \h^*$.
\end{theorem}

\begin{proof}[Proof of Lemma \ref{l:many facing hyperplanes}]
        It is convenient to prove the lemma in terms of the halfspaces.
   Suppose $(\A,\B,\C)$ is the facing triple of halfspaces.
    Let $\A_1:=\A$ and $\B_1:=\B$.
    We start by applying an automorphism to the pair $(\A_1,\B_1)$ that flips the halfspace $\C^*$ to get the pair $(\A_2,\B_2)$.
    Clearly the $\{\A_1,\B_1,\A_2,\B_2\}$ are facing halfspaces and the pair $\{\hA_2,\hB_2\}$ is strongly separated.
    By induction, suppose we have constructed the $2(n-1)$ facing halfspaces $\{\A_1,\B_1,\A_2,\B_2,\ldots,\A_{n-1},\B_{n-1}\}$ such that $\hA_i$ and $\hB_i$ are strongly separated for each $i$ where $n\geq 2$.
    To construct $2n$ facing hyperplanes, we apply an automorphism to the triple  $(\A_1,\B_1,\A_{n-1})$ that flips $\B^*_{n-1}$ to get $(\A'_1,\B'_1,\A'_{n-1})$ 
      and then apply another automorphism to the pair $(\A'_1,\B'_1)$ that flips $\A'^*_{n-1}$ to get another pair $(\A''_1,\B''_1)$ [See figure \ref{f:flipping hyperplanes}].
    By construction, $\{\A'_1,\B'_1\}$ and $\{\A''_1,\B''_1\}$ are both pairs of strongly separated hyperplanes.
    Finally, we claim that the collection $\{\A_1,\B_1,\ldots,\A_{n-2},\B_{n-2},\A'_1,\B'_1,\A''_1,\B''_1\}$ consists of facing halfspaces [Figure \ref{f:flipping hyperplanes}].
    \medskip

    First, we observe that halfspaces in the collection $\{\A_1',\B_1',\A''_2,\B''_2\}$ are mutually disjoint. This is because $\A_1'$ and $\B_1'$ are contained in $\A'^*_{n-1}$, and the pair $(\A''_1,\B''_1)$ is obtained by applying an automorphism to the pair $(\A'_1,\B'_1)$ that flips $\A'^*_{n-1}$. It follows that both $\A''_1$ and $\B''_1$ are contained inside $\A'_{n-1}$. The claim follows.
    \medskip
Next we show that both $\A'_1$ and $\B'_1$ are disjoint from the collection $\{\A_1,\B_1,\ldots,\A_{n-2},\newline \B_{n-2}\}$.
By construction, $(\A'_1,\B'_1)$ is obtained by applying an automorphism to the pair  $(\A_1,\B_1)$ that flips $\B^*_{n-1}$, therefore both $\A'_1$ and $\B'_1$ are  contained inside $\B_{n-1}$.
Since $\B_{n-1}$ is disjoint from any set from the collection $\{\A_1,\B_1,\ldots,\A_{n-2},\B_{n-2}\}$, so are both $\A'_1$ and $\B'_1$.
    \medskip

Similarly, we can show  both $\A''_1$ and $\B''_1$ are disjoint from the collection $\{\A_1,\B_1,\ldots, \newline \A_{n-2},\B_{n-2}\}$.   
    First note that, $\A'_{n-1}\subset \B_{n-1}$ by construction.
    Also, $\A''_1\subset \A'_{n-1}$ and $\B''_1\subset \A'_{n-1}$.
    Consequently both $\A''_1$ and $\B''_1$ are contained inside $\B_{n-1}$ which is disjoint from any set from the collection  $\{\A_1,\B_1,\ldots,\A_{n-2},\B_{n-2}\}$. This finishes the proof.   
   \end{proof}

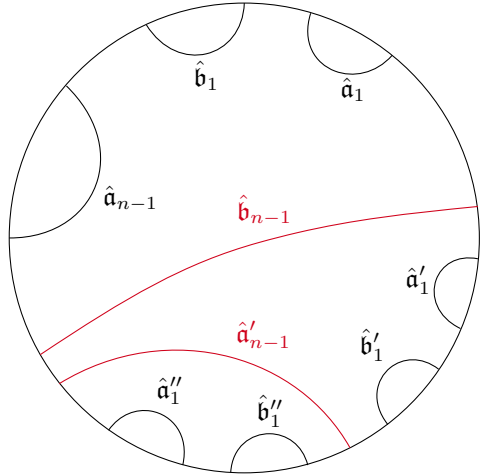
\begin{figure}[htp]\label{f:flipping hyperplanes}
    \centering
    \begin{tikzpicture}[x=0.55pt,y=0.55pt,yscale=-1,xscale=1]

\draw   (169,198.67) .. controls (169,109.38) and (241.38,37) .. (330.67,37) .. controls (419.95,37) and (492.33,109.38) .. (492.33,198.67) .. controls (492.33,287.95) and (419.95,360.33) .. (330.67,360.33) .. controls (241.38,360.33) and (169,287.95) .. (169,198.67) -- cycle ;
\draw [color={rgb, 255:red, 208; green, 2; blue, 27 }  ,draw opacity=1 ]   (190.33,279) .. controls (250.33,239) and (289.45,216.1) .. (340.31,201.92) .. controls (349.56,199.33) and (358.43,197.07) .. (367.11,195.05) .. controls (406.14,185.97) and (441.43,181.91) .. (491.33,177) ;
\draw [color={rgb, 255:red, 208; green, 2; blue, 27 }  ,draw opacity=1 ]   (203.33,299) .. controls (240.83,275.87) and (283.68,270.7) .. (321.4,280.74) .. controls (355.44,289.79) and (385.31,311.22) .. (403.33,343) ;
\draw    (321.33,360) .. controls (322.41,342.26) and (335.33,333.19) .. (348.29,333.41) .. controls (359.45,333.6) and (370.63,340.66) .. (374.33,355) ;
\draw    (237.33,331) .. controls (245.12,319.93) and (256.44,316.09) .. (266.54,317.82) .. controls (281.06,320.31) and (293.05,334.35) .. (288.33,355) ;
\draw    (479.33,261) .. controls (463.03,253.11) and (458.62,238.59) .. (462.61,227.46) .. controls (466.21,217.43) and (476.64,210.16) .. (491.33,213) ;
\draw    (429.33,327) .. controls (417.39,311.8) and (421.07,296) .. (431.09,287.94) .. controls (439.53,281.16) and (452.45,279.86) .. (464.33,289) ;
\draw    (169,198.67) .. controls (201.22,198.85) and (223.61,181.24) .. (229.98,157.9) .. controls (235.34,138.26) and (229.36,114.57) .. (208.33,94) ;
\draw    (263.33,52) .. controls (274.28,69.42) and (290.68,75.19) .. (304.55,71.92) .. controls (318.56,68.62) and (330,56.09) .. (330.67,37) ;
\draw    (376.33,44) .. controls (369.99,64.09) and (379.3,78.59) .. (393.03,83.81) .. controls (405.27,88.45) and (421.02,85.73) .. (432.33,73) ;

\draw (395.03,86.81) node [anchor=north west][inner sep=0.75pt]   [align=left] {$\displaystyle \hat{\mathfrak{a}}_{1}$};
\draw (304.55,74.92) node [anchor=north] [inner sep=0.75pt]   [align=left] {$\displaystyle \hat{\mathfrak{b}}_{1}$};
\draw (231.98,160.9) node [anchor=north west][inner sep=0.75pt]   [align=left] {$\displaystyle \hat{\mathfrak{a}}_{n-1}$};
\draw (365.11,192.05) node [anchor=south east] [inner sep=0.75pt]  [color={rgb, 255:red, 208; green, 2; blue, 27 }  ,opacity=1 ] [align=left] {$\displaystyle \hat{\mathfrak{b}}_{n-1}$};
\draw (323.4,277.74) node [anchor=south west] [inner sep=0.75pt]  [color={rgb, 255:red, 208; green, 2; blue, 27 }  ,opacity=1 ] [align=left] {$\displaystyle \hat{\mathfrak{a}} '_{n-1}$};
\draw (268.54,314.82) node [anchor=south west] [inner sep=0.75pt]   [align=left] {$\displaystyle \hat{\mathfrak{a}} ''_{1}$};
\draw (348.29,330.41) node [anchor=south] [inner sep=0.75pt]   [align=left] {$\displaystyle \hat{\mathfrak{b}} ''_{1}$};
\draw (460.61,227.46) node [anchor=east] [inner sep=0.75pt]   [align=left] {$\displaystyle \hat{\mathfrak{a}} '_{1}$};
\draw (429.09,284.94) node [anchor=south east] [inner sep=0.75pt]   [align=left] {$\displaystyle \hat{\mathfrak{b}} '_{1}$};
\end{tikzpicture}
\caption{The above picture shows how to increase the number of facing hyperplanes. The input is four facing hyperplanes $\{\hA_1,\hB_1,\hA_{n-1},\hB_{n-1}\}$ and the output is six facing hyperplanes $\{\hA_1,\hB_1,\hA'_1,\hB'_1, \hA''_1,\hB''_1\}$.}
\end{figure}

As mentioned before, we need to find two regular elements whose axis in each irreducible factor do not intersect each other at infinity.
More precisely, we want the poles of these two elements to be `disjoint' from each other in each factor.  
Recall that poles are defined in terms of descending sequences of half spaces.
We say, two descending sequences of half spaces $\{P_n\}$ and $\{Q_n\}$ are \textbf{disjoint} if $P_m\cap Q_m=\emptyset$ for some $m$.
Otherwise, we say that $\{P_n\}$ and $\{Q_n\}$  \textbf{intersects} each other.
The next lemma shows that a descending sequence of strongly separated halfspaces cannot intersect two disjoint descending sequence of halfspaces simultaneously.
The statement and its proof are similar to \cite[Lemma 5.11]{Fernos2016Random}, but somewhat different, so we include them for the convenience of the reader.

\begin{lemma}[Disjoint descending sequence of halfspaces]\label{l:intersecting hyperspaces}
    Let $\{\A_n\}$ and $\{\B_n\}$ be two disjoint descending sequence of halfspaces in a finite dimensional $\cat$ cube complex $X$.
    Let $\{\C_n\}$ be another descending sequence of strongly separated halfspaces in $X$ that intersects $\{\A_n\}$. Then $\{\C_n\}$ and $\{\B_n\}$ are disjoint.
\end{lemma}
\begin{proof}
 We need to show that $\B_n\cap \C_n=\emptyset$ for sufficiently large $n$.
    By hypothesis, there exist $N$ such that $\A_i\cap \B_i=\emptyset $ for all $i\geq N$.
    Since $\{\C_n\}$ is descending and intersects $\{\A_n\}$, there exists $M$ such that  $\A_N\cap \C_M\neq \emptyset$.
    So, either  $\C_M\subset \A_N$, or $\C^*_M\subset \A_N$ or $\hC_M\pitchfork \hA_N$. 
    \medskip
    
    Next, we will show how to rule out the second and third cases by taking $M$ large enough.
    To rule out the second case, we observe that
    $\C^*_M\subset \C^*_{M+k}\subset \A_N$ for all $k\geq 0$.
    However, since $X$ is finite dimensional, there can be at most finitely many hyperplanes between any two halfspaces.
    So, for large enough $M$, the second case does not occur.
    To rule out the third case, suppose $\hC_M\pitchfork \hA_N$ for some $M$.
    Since $\C_M$ and $\C_{M+1}$ are strongly separated, $\A_N$ cannot be transverse to $\C_{M+1}$.
    So, either $\A_N\cap \C_{M+1}=\emptyset$ or $\C_{M+1}\subset \A_N$.
    By hypothesis $\A_N\cap \C_{M+1}\neq \emptyset$. Therefore, we have $\C_{M+1}\subset \A_N$.
    \medskip
    
    So we are left with the first possibility which is $\C_M\subset \A_N$ for some $M$. Since $\A_N\cap \B_N=\emptyset$, it follows that $\C_{M}\cap \B_N=\emptyset$. Since the sequences of halfspaces are descending, we get $\{\C_n\}$ is disjoint from $\{\B_n\}$.
\end{proof}

We say that a $\cat$ cube complex $X$ is $\textbf{$\R$-like}$ if there is an $\aut(X)$-invariant geodesic line in $X$. The final ingredient to our main theorem is the following Lemma which characterizes cube complexes with invariant Euclidean flats.

\begin{lemma}[\cite{Caprace_2011}]\label{invariant flat}  Let $X$ be a finite dimensional $\cat$ cube complex such that $\aut(X)$ acts essentially. Then $\aut(X)$ stabilizes some $n$-dimensional flat $\R^n\subset X$ if and only if $X$ decomposes as a product $X=X_1\times X_2\times\ldots \times X_n$ of subcomplexes, each of which is essential and $\R$-like.
\end{lemma}
We are now ready to prove our main theorem which appears as Theorem~\ref{intro:v1} in the introduction.

\begin{theorem}[Girth Alternative for lattices, version 1]\label{premain theorem}
     Let $X = X_1 \times \dots \times X_n$ be a product of $n$ irreducible, unbounded, locally compact $\cat$ cube complexes such that $\aut(X_i)$ acts cocompactly and essentially on $X_i$ for all $i$. 
\medskip

Then for any (possibly non-uniform) lattice $\Gamma \leq \aut(X)$, 
either $\Gamma$ virtually fixes some point in $X\cup \bdry X$ or 
$\Gir(\Gamma)=\infty$.
\end{theorem}

\begin{proof}
    Suppose $\Gamma\le\aut(X)$ does not fix a point in $X\cup\bdry X$. We want to show that $\Gir(\Gamma)=\infty$.
    \medskip
    
    Since $\Gamma$ does not have fixed points at infinity, $\aut(X)$ does not have fixed points at infinity. Furthermore, since $\aut(X)$ acts essentially on $X$, we can invoke \cite[Proposition 5.1]{Caprace_2011} to obtain a pair of strongly separated hyperplane $\{\hh_i,\hk_i\}$ in each $X_i$.
    \medskip

    We now apply Theorem \ref{Regular elements in lattices} to obtain an element $\gs\in \Gamma$ that double skewers $(\hh_i,\hk_i)$ in each $X_i$.  
    Let $\h_i$ be the halfspace corresponding to the hyperplane $\hh_i$ such that $\gs^{j+1}\h_i\subset \gs^j\h_i$ for any $j\in \mathbb{N}$.
    Since $(\hh_i,\hk_i)$ are strongly separated hyperplanes,
    it follows that $\{\gs^n\h_i\}$ and $\{\gs^{-n}\h^*_i\}$ are two disjoint descending sequence of strongly separated halfspaces.
    In other words, they are the poles of $\gs$ in $X_i$ (see preliminaries).
    \medskip

    Next, we choose an arbitrary generating set $\{\gamma_1,\gamma_2,\ldots, \gamma_k\}$ of $\Gamma$.
   We will consider the translate of poles of $\gs$ in each $X_i$ under this generating set.
   More precisely, we consider the following list of descending sequences of  halfspaces:
        \begin{equation*}\label{list of poles}
        \{\gs^n\h_i\}, \{\gs^{-n}\h_i^*\}, \{\gamma_j\gs^n\h_i\}, \{\gamma_j\gs^{-n}\h^*_i\}, \{\gamma_j^{-1}\gs^n\h_i\}, \{\gamma_j^{-1}\gs^{-n}\h_i^*\} \tag{$\dagger$}
        \end{equation*}
        where $1\leq i\leq n$ and $1\leq j\leq k$.
Since isometries preserve strong separation, all the sequences in the above list consists of strongly separated halfspaces.
        
        (Also note that, for a fix $i$, all the halfspaces in the above list may not live in $X_i$, because $\gamma_j$ can switch isomorphic factors. However, they all live in some $X_i$). 
    \medskip
    
    Now, we are going to produce another regular element $\tau\in \Gamma$ such that the poles in each $X_i$ are disjoint from the list (\ref{list of poles}). In particular, its poles will be disjoint from that of $\gs$.
    This is where we are going to use the fact that $\Gamma$ does not fix a point in $X\cup\bdry X$.
    This will ensure that each $X_i$ is large, in a certain sense. More precisely, we will see that each $X_i$ has a facing triple.
    
    First we observe that, since each $X_i$ is finite dimensional and locally compact and $\aut(X_i)$ acts essentially and cocompactly on $X_i$, by \cite[Theorem 7.2]{Caprace_2011}, either $\aut(X_i)$ stabilizes some Euclidean flat or there is a facing triple of hyperplanes in $X_i$.

     Suppose $\aut(X_i)$ stabilizes some Euclidean flat $F_i\subset X_i$. 
     Since $X_i$ is irreducible, it follows from Lemma~\ref{invariant flat} that $F_i$ is $\textbf{R}$-like.
     Thus $\aut(X_i)$ has an index two subgroup which fixes a point at infinity contradicting our assumption on $\Gamma$.
     Therefore, we have a facing triple of hyperplanes in each $X_i$.
    Moreover, we can choose this facing triple in a way so that there is a pair of hyperplanes in the triple that are strongly separated (see the proof of \cite[Theorem 7.2]{Caprace_2011} for details).

    By Lemma \ref{l:many facing hyperplanes}, there exist infinitely many pairs of strongly separated hyperplanes in each $X_i$ that are facing each other.
    Given a pair of strongly separated hyperplanes taken from each $X_i$, we can choose an element from $\Gamma$ that double skewers each pair in the corresponding $X_i$ due to Theorem \ref{Regular elements in lattices}.
    Since all the pairs are facing each other, the poles of these group elements give us infinitely many mutually disjoint descending sequence of strongly separated hyperplanes in $X_i$. 
Of these infinitely many descending sequences of halfspaces in $X_i$, all but finitely many are disjoint from the list (\ref{list of poles})  by Lemma \ref{l:intersecting hyperspaces} because the list (\ref{list of poles}) contains a finite number of descending sequences of strongly separated halfspaces.
In particular, there exists a pair of disjoint halfspaces $\{\h'_i,\mathfrak{d}'_i\}$ in each $X_i$, such that for large enough $N$, both $\h'_i$ and $\mathfrak{d}'_i$ are disjoint from the following collection 

\[
\bigl\{\gs^{N}\h_i, \gs^{-N}\h_i,\gamma_j\gs^{N}\h_i^*, \gamma_j\gs^{-N}\h^*_i, \gamma_j^{-1}\gs^N\h_i, \gamma_j^{-1}\gs^{-N}\h_i^*\mid 1\leq i\leq n,1\leq j\leq k\bigr\}.
\]

Let $\tau\in \Gamma$ be an element that double skewers $(\hh'_i,\hk'_i)$ in each $X_i$.
Without loss of generality, let us assume that $\{\tau^n\h'_i\}$ and $\{\tau^{-n}\h'^*_i\}$ are the poles of $\tau$ in $X_i$.
\medskip

We let $Y:=\sqcup_{i=1}^n X_i$. Note that $\Gamma$ acts on $Y$. 
We set $U_{\gs'}:=\bigcup_{i} (\gs^N\h_i\cup \gs^{-N}\h_i^*)$ and $U_{\tau'}:=\bigcup_{i}(\tau^N\h'_i\cup \tau^{-N}\h'^*_i)$.
\medskip

Note that, the complement of $\gs^N\h_i\cup \gs^{-N}\h_i^*$ in $X_i$ get sent to $\gs^N\h_i\cup \gs^{-N}\h_i^*$ under high enough power of $\gs^{\pm 1}$.
Taking union over $i$, we obtain that complement of $U_\gs$ in $Y$ get sent to $U_\gs$ under high enough power of $\gs^{\pm 1}$. Similarly, complement of $U_\tau$ in $Y$ get sent to $U_\tau$ under high enough power of $\tau^{\pm 1}$. 
We set $\gs'=\gs^M$, $\tau'=\tau^M$, 
where we choose $M$  so large that for all $i\in \mathbb{Z}-\{0\}$, we have 
\begin{equation}\label{gs}
      (\gs')^i(Y-U_{\gs'})\subset U_{\gs'} \quad \text{    and    }\quad 
     (\tau')^i(Y-U_{\tau'})\subset U_{\tau'} \tag{$1$}
\end{equation}

Next, we observe that, by construction,  $U_{\gs'}$ and its translates under the generating set and their inverses do not intersect $U_{\tau'}$. In other words we have,
\begin{equation}\label{disjoint open sets 1}
    U_{\gs'}\cup \bigcup_{\epsilon=\pm 1}\bigcup_{j=1}^k\gamma_j^\epsilon (U_{\gs'})\subset Y-U_{\tau'} \tag{$2$}.
    \end{equation}
Since $\gamma_j(U_{\gs'})\cap U_{\tau'}=\emptyset$ iff $\gamma^{-1}_j(U_{\tau'})\cap U_{\gs'}=\emptyset$ for any $\gamma_j$, it follows from \eqref{disjoint open sets 1} that $U_{\tau'}$ and its translates under the generating set and their inverses do not intersect $U_{\gs'}$.
In other words, we have
\begin{equation}\label{disjoint open sets 2}
    U_{\tau'}\cup \bigcup_{\epsilon=\pm 1}\bigcup_{j=1}^k\gamma_j^\epsilon (U_{\tau'})\subset Y-U_{\gs'} \tag{$3$}
\end{equation}

Furthermore, we can assume (by taking $N$ large enough) that the following set
\[Y-U_{\sigma'}\cup U_{\tau'}\cup\bigcup_{\varepsilon=\pm1}\bigcup_{j=1}^k\gamma_j^\varepsilon(U_{\sigma'}\cup U_{\tau'})
\]
is nonempty and 
for convenience, let us call this set $B$.
\medskip
 
 We claim that $\gs',\tau',U_{\gs'},U_{\tau'}$, and any point $x\in B$ satisfies the three properties of Lemma \ref{t:criteria for infinite girth}.
 \medskip
 
 The first property is immediate by our choice of $x$.
 To check the second property, note that $x\in Y-U_{\gs'}$ and together with \eqref{disjoint open sets 2} this implies
\[
 \{x\}\cup U_{\tau'}\cup \bigcup_{\epsilon=\pm 1}\bigcup_{j=1}^k\gamma_j^\epsilon (U_{\tau'})\subset Y-U_{\gs'} 
 \] 
  Therefore, \eqref{gs} gives us
 \[
 (\gs')^i(\{x\}\cup U_{\tau'}\cup \bigcup_{\epsilon=\pm 1}\bigcup_{j=1}^k\gamma_j^\epsilon (U_{\tau'}))\subset U_{\gs'} \quad \text{for all } i\in \mathbb{Z}-\{0\}
 \]
 which is the desired second property.
 \medskip
 
\noindent Similarly, $x\in Y-U_{\tau'}$, together with \eqref{disjoint open sets 1} and \eqref{gs} gives us
 \[
 (\tau')^i(\{x\}\cup U_{\gs'}\cup \bigcup_{\epsilon=\pm 1}\bigcup_{j=1}^k\gamma_j^\epsilon (U_{\gs'}))\subset U_{\tau'}
 \]
 which is the desired third property.

 \noindent Therefore, $\Gir(\Gamma)=\infty$ by Theorem \ref{t:criteria for infinite girth}.
 \medskip
 \end{proof}
We will now prove the following version of Girth Alternative. The statement and its proof are inspired by  \cite[Corollary G]{Caprace_2011}.  

\begin{theorem}[Girth Alternative for lattices, version 2]\label{Girth alternative v2}
Let $X = X_1 \times \dots \times X_n$ be a product of $n$ irreducible, unbounded, and locally compact $\cat$ cube complexes such that $\aut(X_i)$ acts cocompactly and essentially on $X_i$ for all $i$.       
      
      Then for any (possibly non-uniform) lattice $\Gamma \leq \aut(X)$ acting properly on $X$, 
either $\Gamma$ is \{locally finite\}-by-\{virtually abelian\} or 
$\Gir(\Gamma)=\infty$.
\end{theorem}
The proof of the above will rely on the proof of Theorem~\ref{premain theorem} and a result of \cite{CL10}.
In addition, we will need the following lemmas to equivariantly triangulate  certain subcomplexes in $X$.
\begin{lemma}\label{l:positive translation length}
    Suppose $G$ acts by automorphism on a locally finite cube complex $X$. 
    Suppose $Y\subset X$ be a $G$-invariant subset and $G$ does not fix $Y$ point-wise.
    Then there exists  $p\in Y$ such that
    \[k:=\inf\{d(p,g\cdot p)\mid g\in G,p\neq g\cdot p\}>0
    \]
\end{lemma}
\begin{proof}
By the assumption, we can choose $p\in Y$ such that $g\cdot p\neq p$ for some $g\in G$. Now we have two cases to consider.

Suppose $p$ is a vertex in $X$. Since $G$ acts by automorphisms, translates of $p$ can only be other vertices. It follows that $k$ is at least the minimum of the distances between $p$ and other vertices. Since $X$ is locally finite, it follows that $k>0$.

  If $p$ is not a vertex, then $p$ is in the interior of some cube $C$ in $X$. The infimum of distances between $p$ and all the translates of $p$ that are inside a different cube is clearly positive.
  Since there are only finitely many automophisms of $C$, there are only finitely many nontrivial translates of $p$ that stay inside $C$. Therefore the minimum distance from $p$ to any of its translates inside $C$ is also positive. The claim follows.
\end{proof}

\begin{lemma}\label{l:translation length to triangulation}
    Suppose $G$ acts on $\mathbb{R}$ by isometries such that for some $p$
    \[k:=\inf\{d(p,g\cdot p)\mid g\in G,p\neq g\cdot p\}>0.
    \]
    Then there exists a finite index subgroup of $H$ and an $H$-equivariant triangulation of $\mathbb{R}$. 
\end{lemma}

\begin{proof}
Let $H$ be the  index two subgroup of $G$ consisting of only the orientation preserving isometries.
In particular, $H$ acts by translation on $\mathbb{R}$ and hence is a subgroup of $\mathbb{R}$.
Either $H$ is trivial or $k>0$ implies that
\[
k':=\inf\{d(p,h\cdot p)\mid h\in H,p\neq h\cdot p\}>0
\]
It follows that $H$ is generated by the isometry that translates by $k'$ amount.
Therefore,  placing the vertices at integer multiples of $k'$, we obtain a $H$-equivariant triangulation of $\mathbb{R}$.
\end{proof}
Combining the above two lemmas we immediately get the following.
\begin{lemma}\label{cubification}
    Suppose $G$ acts by automorphisms on a locally finite $\R$-like cube complex $X$ and let $L
\subset X$ be an $\aut(X)$-invariant bi-infinite geodesic in $X$. Then there exists a finite index subgroup $H$ of $G$ and a $H$-equivariant triangulation of $L$.
\end{lemma}
\begin{proof}
    If $G$ fixes $L$ pointwise, then any triangulation of $L$ is $G$-equivariant.
    Otherwise, we can apply  Lemma~\ref{l:positive translation length} with $Y:=L$ to conclude that $k$ as in Lemma~\ref{l:positive translation length} is positive. 
    Therefore we can apply Lemma~\ref{l:translation length to triangulation} to $L$ which yields a finite index subgroup $H$ of $G$ and an $H$-equivariant triangulation of $L$.
\end{proof}

\begin{proof}[Proof of Theorem~\ref{Girth alternative v2}]
    It follows from \cite[Theorem~1.7]{CL10} that $\Gamma$ is \{locally finite\}-by-\{virtually abelian\} if and only if $\Gamma$ is amenable.
    Thus it is enough to show that if $\Gamma$ is not amenable, then $\Gir(\Gamma)=\infty$.
    Suppose $\Gamma$ is non-amenable.
    We claim there is an irreducible factor of $X$ that is not $\R$-like.

  If each irreducible factor of $X$ is $\R$-like, then we can take the $\aut(X_i)$-invariant geodesic in each factor and take their product to get an invariant $n$-dimensional Euclidean flat.
  Let $\Gamma'$ be the finite index subgroup of $\Gamma$ that stabilizes each factor.
  In particular, $\Gamma'$ stabilizes the bi-infinite geodesic lines in each factor.
  Applying Lemma~\ref{cubification} to the action of $\Gamma'$ on each $\R$-like factor, we get a finite index subgroup of $\Gamma'$ which acts by automorphisms on some triangulation of the bi-infinite geodesic line in that $\R$-like factor. Taking the intersection of all these finite index subgroup for each factor we get another finite index subgroup $\Gamma''\le\Gamma'$ that acts by automorphisms on each triangulated geodesic line. $\Gamma''$ therefore acts by automorphisms on a Euclidean flat $Y$ where $Y$ is the product of triangulated geodesic lines.
 By \cite[Lemma 2.8]{Genevois}, the set of bounded components in the Roller boundary of $Y$ is finite and nonempty.
 These bounded components have cubical structure that is respected by the induced action of $\Gamma''$.
  Furthermore, the action of $\Gamma''$ stabilizes the union of these bounded components.
  In particular, $\Gamma''$ stabilizes some finite set of vertices in the Roller boundary and consequently a finite index subgroup of $\Gamma''$ fixes a point in the Roller boundary of $Y$.
  By \cite[Theorem A.5]{CL}, $\Gamma''$ is amenable.
  Since $\Gamma''\le\Gamma$ is a finite index subgroup, $\Gamma$ is amenable.
  This is a contradiction. 
  
    Hence at least one irreducible factor of $X$ is not $\R$-like, and hence contains a facing triple by~\cite[Theorem 7.2]{Caprace_2011}.
    Now, we can discard all the $\R$-like factors from $X$ to get a space $Y$  such that each irreducible factor of $Y$ has a facing triple.
    Now we can run the same argument as in Theorem~\ref{premain theorem} to conclude that the girth of $\Gamma$ is infinite.
\end{proof}

\begin{remark}\label{lattice vs non-lattice}
In our Girth Alternative theorems, we require the group to be a lattice in $\aut(X)$ whereas the analogous Tits Alternative theorems of Caprace--Sageev \cite{Caprace_2011} (also see \cite{sageev_wise_2005}) do not require the group to be a lattice. 
The main reason for the lattice assumption in our case is that we need regular elements that double skewer a given pair of hyperplanes in each irreducible factor to obtain infinite girth. 
And Theorem \ref{Regular elements in lattices} of \cite{Caprace_2011} gives us such elements for lattices.
Whereas, in the proof of Tits Alternative, one needs to find a free subgroup.  
In this case, we only need two group elements that double skewer distinct pairs of hyperplanes in some irreducible factor (see \cite{Caprace_2011} for details).
Interestingly, authors in \cite[Theorem 1.5]{Fernos2016Random} proved that regular elements exist whenever the group acts essentially and without fixed points in $X\cup\partial_\infty X$. Moreover, the group does not need to be a lattice in order to contain regular elements. 
This suggests that the lattice assumption may be dropped from Theorem~\ref{premain theorem}.
 However, we do not know how to use the probabilistic methods of \cite{Fernos2016Random} to produce an abundance of regular elements concretely without the lattice assumption.
\end{remark}

\begin{remark}\label{lattice vs non-lattice2}
    Another way of getting rid of the lattice assumption is the following. Note that $\Gamma$ has a finite index subgroup $\Gamma'$ that stabilizes each irreducible factor factor by Proposition~\ref{p:product decomposition}. With that in mind, we can run the same argument as in Theorem~\ref{premain theorem}, but only working with the $\Gamma'$-action on $X_1$ to prove that either $\Gamma$ virtually fixes a point or $\Gir(\Gamma')=\infty$ without assuming $\Gamma$ to be lattice: since we can forget about the other factors, we only need $\sigma$ and $\tau$ to act as contracting isometries on $X_1$, and such elements always exist without the lattice assumption by \cite{Caprace_2011}.
    So, another way to remove the lattice assumption in Theorem~\ref{premain theorem} would be to show that  $\Gir(\Gamma')=\infty$ implies $\Gir(\Gamma)=\infty$. 
     Although it seems plausible, we do not know whether containing a finite index subgroup with infinite girth is enough to guarantee the group itself has infinite girth.
\end{remark}

\section{Curious Examples}\label{section:counterexample}

As promised in the introduction, we construct spaces which are \textbf{R}-like with particular group actions showing that Corollary~\ref{intro:v3} does not hold if both the assumptions properness and having a bound on the order of finite subgroups are omitted.
\medskip

For a cube complex $X$, we call a pair of vertices $v,v^*\in X$ {\bf diametrically opposed} if there exists a labeling of the halfspaces of $X$ so that $\theta(v)=\{\h\;|\;\hat{\h}\in\hat{\mathcal{H}}\}$ and $\theta(v^*)=\{\h^*\;|\;\hat{\h}\in\hat{\mathcal{H}}\}$. Note that a CAT(0) cube complex $X$ contains diametrically opposed vertices if and only if there is a geodesic path between vertices that crosses every hyperplane in $X$.
\medskip

If $X$ and $Y$ are CAT(0) cube complexes each containing a pair of diametrically opposed vertices, then their product $X\times Y$ will contain a pair of diametrically opposed vertices. Indeed, if $\{\h_i\}$ and $\{\h_i^*\}$ define vertices $x,x^*\in X$, and $\{\mathfrak{k}_i\}$ and $\{\mathfrak{k}_i^*\}$ define vertices $y,y^*\in Y$, then $(x,y)$ and $(x^*,y^*)$ give diametrically opposed vertices in $X\times Y$.

\begin{definition}[Cube Complex $\mathcal{L}{[X]}$]
    Begin with a CAT(0) cube complex $X$ which contains diametrically opposed vertices $v$ and $v^*$. For all $i\in\mathbb{Z}$, let $X_i$ be an isomorphic copy of $X$, where $v_i$ and $v_i^*$ denote the images of $v$ and $v^*$, respectively. The cube complex $\mathcal{L}[X]$ is then formed by identifying $v_i^*$ with $v_{i+1}$ for all $i\in\mathbb{Z}$. We think of $\mathcal{L}[X]$ informally as a ``line of $X$'s."
\end{definition}
\medskip

The space $\mathcal{L}[X]$ is clearly a cube complex as the map attaching distinct copies of $X$ is an isometry of $v$. That $\mathcal{L}[X]$ is locally CAT(0) is also clear, since the link of a vertex labeled $v$ (or $v^*$) in $\mathcal{L}[X]$ is isomorphic to the disjoint union of two copies of the link of $v$ in $X$. Lastly, $\mathcal{L}[X]$ is simply-connected by construction.
\medskip

There is a natural action of $\mathbb{Z}$ on $\mathcal{L}[X]$ whereby each copy $X$ is shifted to an adjacent copy. With $\aut_{v,v^*}(X)\le \aut(X)$ denoting the automorpisms fixing both $v$ and $v^*$, $\Sigma_X=\oplus \aut_{v,v^*}(X)$ acts on $\mathcal{L}[X]$ by allowing only finitely many coordinate groups to act nontrivially at a time. Clearly $\mathbb{Z}\cap \Sigma_X$ is the identity element. We are interested in the semi-direct product $\mathbb{Z}\ltimes \Sigma_X\cong\aut_{v,v^*}(X)\wr\mathbb{Z}$, where multiplication is defined as $(\sigma_1,\tau^m)(\sigma_2,\tau^n)\;=\;(\sigma_1\tau^m\sigma_2\tau^{-m},\tau^{m+n})$ and $(\sigma,\tau^n)^{-1}=(\tau^{-n}\sigma^{-1}\tau^n,\tau^{-n})$. 
Note that $\mathbb{Z}\ltimes\Sigma_X$ is generated by the group $\aut_{v,v^*}(X)$ and a single shift. 
\medskip

We define a group action of $\mathbb{Z}\ltimes\Sigma_X$ on $\mathcal{L}[X]$ by 
\[
    (\sigma,\tau^n)\cdot (x)\;=\;\sigma\tau^n\cdot(x)
\]
 
This action is faithful since $(\sigma,\tau^n)$ acts trivially if and only if both $\sigma$ and $\tau$ are the identity element in $\Sigma_X$ and $\mathbb{Z}$, respectively. Since it acts faithfully, there is an injective map $\mathbb{Z}\ltimes\Sigma_X\to\aut(\mathcal{L}[X])$ permitting us to think of $\mathbb{Z}\ltimes\Sigma_X$ as a subgroup of the automorphism group.
\medskip

We highlight two additional properties of $\mathbb{Z}\ltimes\Sigma_X$. It acts cocompactly on $\mathcal{L}[X]$ as any copy of $X$ gives a compact set whose translates cover $\mathcal{L}[X]$. And, since any copy of $X$ can be mapped to any other copy of $X$, the action of $\mathbb{Z}\ltimes\Sigma_X$ on $\mathcal{L}[X]$ is essential.
Next, we will show that $\mathbb{Z}\ltimes\Sigma_X$ satisfies a law and thus, for nontrivial $\Sigma_X$, has finite girth.
\medskip

Recall that a group $\Gamma$ {\bf satisfies a law} if there is a word $w(x_1,x_2,\ldots,x_n)$ on $n$ letters such that $w(\gamma_1,\gamma_2,\ldots,\gamma_n)=1$ in $\Gamma$ for any $\gamma_1,\gamma_2,\ldots,\gamma_n\in \Gamma$. That $\mathbb{Z}\ltimes\Sigma_X$ satisfies a law can be shown from observing that both $\Sigma_X$ and ${\mathbb{Z}\ltimes\Sigma_X}/\Sigma_X\cong\mathbb{Z}$ satisfy a law. However, the space $\mathcal{L}[X]$ provides a geometric proof which we provide.

\begin{lemma}\label{lem: satisifes law}
    For a compact $\cat$ cube complex $X$ which has diametrically opposed vertices, the group $\mathbb{Z}\ltimes\Sigma_X$ satisfies a law.
\end{lemma}

\begin{proof}
     Let $|\aut(X)|=n$ and $x=(\sigma_1,\tau^k)$ and $y=(\sigma_2,\tau^l)$ be elements in $\mathbb{Z}\ltimes\Sigma_X$. Note that the commutator $[x,y]$ acting on $\mathcal{L}[X]$ maps each copy of $X$ to itself. Clearly then, $[x,y]^{n}$ acts trivially. Since the action is faithful, $[x,y]^{n}$ must then be the identity.
\end{proof}

Although a sketch of the proof for the following result can be found elsewhere in the literature (see Proposition 1 in \cite{dyubina2000instability}), we provide a proof for completeness. 

\begin{lemma}\label{lem:direct sum not virt solv}
    If a group $H$ is not solvable, then $\oplus_\mathbb{Z}H$ is not virtually solvable.
\end{lemma}
\begin{proof}
    Let $G\leq\oplus_\mathbb{Z}H$ be a finite index subgroup. We claim that, for some $i$, the projection map $\pi_i:G\rightarrow H$ to the $i^{th}$ factor is surjective. If not, then for each $i$ there is a nontrivial element $a_i$, in the $i^{th}$ factor, with $a_i\notin \pi_i(G)$. Consider the elements $b_i=(\ldots,e, a_i,e,\ldots)$ in the direct sum, where $a_i$ appears in the $i^{th}$ position. It follows that $\pi_i(b_ib_j^{-1})=a_i$ for $i\neq j$, implying that $b_ib_j^{-1}\notin G$ for all such $i,j$. Then, $b_iG=b_jG$ if and only if $b_i=b_j$. Since $b_i$ is nontrivial for all $i\in\mathbb{Z}$, $G$ has countably many distinct cosets and hence cannot be a finite index subgroup. Thus the map $\pi_i:G\rightarrow H$ is a surjection for some $i$, and we conclude that a finite index subgroup $G$ of $\oplus_\mathbb{Z}H$ necessarily surjects onto $H$. Since $H$ is non solvable, $G$ is non solvable.
\end{proof}
Now, consider the subgroup $\Sigma_X\le\mathbb{Z}\ltimes\Sigma_X$. By Lemma \ref{lem:direct sum not virt solv}, if $\aut_{v,v^*}(X)$ is not solvable then $\Sigma_X$ is not virtually solvable, and hence neither is $\mathbb{Z}\ltimes\Sigma_X$. This fact combines with lemma \ref{lem: satisifes law} to give the following.

\begin{example}
    If $\aut_{v,v^*}(X)$ is non-solvable then $\mathbb{Z}\ltimes\Sigma_X$ provides an example of a group acting essentially and cocompactly on a finite dimensional CAT(0) cube complex that is not virtually solvable yet satisfies a law. Consequently, this example shows that the Girth Alternative, as stated in Corollary~\ref{intro:v3}, fails for the class of groups acting cocompactly on finite dimensional CAT(0) cube complexes.
    \medskip

    As a special instance, for $X=I^n$, $\aut_{v,v^*}(X)\cong S_n$. Thus, for $n\geq5$, $\mathbb{Z}\ltimes\Sigma_X\cong S_n\wr\mathbb{Z}$ provides a more concrete example.
\end{example}

\medskip

\section{Questions and Remarks}

\begin{enumerate}
    \item Can we drop the lattice assumption in Theorem \ref{premain theorem} and instead assume that $\Gamma$ acts essentially on $X$?
    
   \noindent As suggested in Remark \ref{lattice vs non-lattice}, an affirmative answer to the above question depends on an affirmative answer to the following question: 
 Can we drop the `lattice' assumption in Theorem \ref{Regular elements in lattices} and instead assume that $\Gamma$ acts essentially on $X$? 

 \item Remark~\ref{lattice vs non-lattice2} motivates us to ask the following question.
 Suppose $X=X_1\times X_2\times \ldots\times X_n$ where $X_i$ are irreducible and $\Gamma\le\aut(X)$. Let $\Gamma'$ be the finite index subgroup of $\Gamma$ that stabilizes each irreducible factor $X_i$.
 If $\Gir(\Gamma')$ is infinite, must $\Gir(\Gamma)$ be infinite?

        \item Can we generalize any of statements of theorem \ref{intro:v2} through \ref{intro:v1} to groups acting on non-locally compact cube complexes? This question is again motivated by the analogous Tits Alternatives of \cite{Caprace_2011}, where local compactness was not required.
    
     \item Are the alternatives in Theorem~\ref{premain theorem} mutually exclusive? In other words, under the hypothesis of Theorem~\ref{premain theorem}, if $\Gamma$ virtually fixes some point in $X\cup \partial_\infty X$, does that imply $\Gir(\Gamma)<\infty$.
\end{enumerate}

\section{Acknowledgements.}

We are very grateful to Azer Akhmedov for the valuable suggestions and comments on the initial draft of the paper. We are thankful to Pierre-Emmanuel Caprace for helpful conversations.

 \printbibliography

@InProceedings{Cornulier_Mann,
author="de Cornulier, Yves
and Mann, Avinoam",
editor="Arzhantseva, Goulnara N.
and Burillo, Jos{\'e}
and Bartholdi, Laurent
and Ventura, Enric",
title="Some Residually Finite Groups Satisfying Laws",
booktitle="Geometric Group Theory",
year="2007",
publisher="Birkh{\"a}user Basel",
address="Basel",
pages="45--50",
}

@book {MR2707536,
    AUTHOR = {Bleak, Collin},
     TITLE = {Solvability in groups of piecewise-linear homeomorphisms of
              the unit interval},
      NOTE = {Thesis (Ph.D.)--State University of New York at Binghamton},
 PUBLISHER = {ProQuest LLC, Ann Arbor, MI},
      YEAR = {2005},
     PAGES = {111},
}

@article {MR0782231,
    AUTHOR = {Brin, Matthew G. and Squier, Craig C.},
     TITLE = {Groups of piecewise linear homeomorphisms of the real line},
   JOURNAL = {Invent. Math.},
    VOLUME = {79},
      YEAR = {1985},
    NUMBER = {3},
     PAGES = {485--498},
}

@article {MR0745513,
    AUTHOR = {Ivanov, N. V.},
     TITLE = {Algebraic properties of the {T}eichm\"uller modular group},
   JOURNAL = {Dokl. Akad. Nauk SSSR},
    VOLUME = {275},
      YEAR = {1984},
    NUMBER = {4},
     PAGES = {786--789},
}

@article{10.1215/S0012-7094-83-05046-9,
author = {Joan S. Birman and Alex  Lubotzky and John McCarthy},
title = {{Abelian and solvable subgroups of the mapping class groups}},
volume = {50},
journal = {Duke Mathematical Journal},
number = {4},
pages = {1107 -- 1120},
year = {1983},
}

@article{https://doi.org/10.1112/jtopol/jts035,
author = {Leary, Ian J.},
title = {A metric Kan–Thurston theorem},
journal = {Journal of Topology},
volume = {6},
number = {1},
pages = {251-284},
year = {2013}
}

@InProceedings{10.1007/978-3-319-43674-6_4,
author="Farley, Daniel",
editor="Davis, Michael W.
and Fowler, James
and Lafont, Jean-Fran{\c{c}}ois
and Leary, Ian J.",
title="A Proof of Sageev's Theorem on Hyperplanes in CAT(0) Cubical Complexes",
booktitle="Topology and Geometric Group Theory",
year="2016",
publisher="Springer International Publishing",
address="Cham",
pages="127--142",

}

@incollection{Gro87,
       author={M. Gromov},
       title={Hyperbolic Groups},
       booktitle={Essays in Group Theory},
       editor={S. Gersten},
       series = {MSRI Publications},
       volume= {8},
       pages={75--263},
       publisher={Springer, New York},
       year={1987}
       
 }

@article{AkhmedovMishra2023,
    author = {Azer Akhmedov and Pratyush Mishra},
    title = {Girth Alternative for HNN extensions},
    journal = {arXiv e-prints} ,
    year = {2023},
    pages={arXiv:2211.13326v3}
}

@article{sageev_wise_2005,
author = {Sageev, Michah and Wise, Daniel T.},
title = {The Tits Alternative for Cat(0) Cubical Complexes},
journal = {Bulletin of the London Mathematical Society},
volume = {37},
number = {5},
pages = {706-710},
year = {2005},
}

@article{Caprace_2011,
  year = {2011},
volume = {21},
  number = {4},
  pages = {851--891},
  author = {Pierre-Emmanuel Caprace and Michah Sageev},
  title = {Rank Rigidity for Cat(0) Cube Complexes},
  journal = {Geometric and Functional Analysis}
}

@article {Fernos2016Random,
    AUTHOR = {Fern\'{o}s, Talia and L\'{e}cureux, Jean and Math\'{e}us, Fr\'{e}d\'{e}ric},
     TITLE = {Random walks and boundaries of {$\rm CAT(0)$} cubical
              complexes},
   JOURNAL = {Comment. Math. Helv.},
    VOLUME = {93},
      YEAR = {2018},
    NUMBER = {2},
     PAGES = {291--333},
}

@article{FERNÓS_2018,
title={The Furstenberg–Poisson boundary and CAT(0) cube complexes}, 
volume={38}, 
number={6}, 
journal={Ergodic Theory and Dynamical Systems}, 
author={Fern\'os, Talia}, 
year={2018}, 
pages={2180–2223},
}

@article{NV_PoissonBdry,
  title={The Poisson boundary of CAT (0) cube complex groups},
  author={Nevo, Amos and Sageev, Michah},
  journal={Groups, Geometry, and Dynamics},
  volume={7},
  number={3},
  pages={653--695},
  year={2013}
}

@article{Roller2016PocSM,
  title={Poc Sets, Median Algebras and Group Actions},
  author={Martin A. Roller},
  journal={arXiv: General Topology},
  year={2016},
}

@article{Haglund2007IsometriesOC,
  title={Isometries of CAT(0) cube complexes are semi-simple},
  author={Fr{\'e}d{\'e}ric Haglund},
  journal={Annales math{\'e}matiques du Qu{\'e}bec},
  year={2007},
  volume={47},
  pages={249 - 261},
}

@article {kar2016ping,
    AUTHOR = {Kar, Aditi and Sageev, Michah},
     TITLE = {Ping pong on {$\rm CAT(0)$} cube complexes},
   JOURNAL = {Comment. Math. Helv.},
    VOLUME = {91},
      YEAR = {2016},
    NUMBER = {3},
     PAGES = {543--561},
}

@article{AKHMEDOV2003198,
title={On the girth of finitely generated groups},
journal={Journal of Algebra},
volume={268},
pages={198-208},
year={2003},
author={Azer Akhmedov},
}

@article{Schleimer02,
author = {Schleimer, Saul},
year = {2001},
month = {02},
pages = {preprint},
title = {On The Girth Of Groups}
}

@article{akhmedov2014girth, 
title={Girth Alternative for subgroups of $PL_o(I)$}, 
journal={Glasgow Mathematical Journal}, author={Akhmedov, Azer}, 
year={2024}, 
pages={1–10}}

@article{akhmedov2005girth,
title = {The girth of groups satisfying Tits Alternative},
journal = {Journal of Algebra},
volume = {287},
number = {2},
pages = {275-282},
year = {2005},
author = {Azer Akhmedov},
}

@article {nakamura2011girth,
    AUTHOR = {Nakamura, Kei},
     TITLE = {The girth alternative for mapping class groups},
   JOURNAL = {Groups Geom. Dyn.},
    VOLUME = {8},
      YEAR = {2014},
    NUMBER = {1},
     PAGES = {225--244},
}

@article{yamagata,
author = {Saeko Yamagata},
title = {{The girth of convergence groups and mapping class groups}},
volume = {48},
journal = {Osaka Journal of Mathematics},
number = {1},
publisher = {Osaka University and Osaka Metropolitan University, Departments of Mathematics},
pages = {233 -- 249},
year = {2011},
}

@article{Bestvina1997TheTA2,
  title={The Tits alternative for Out($F_n$) II: A Kolchin type theorem},
  author={Mladen Bestvina and Mark Feighn and Michael Handel},
  journal={Annals of Mathematics},
  year={1997},
  volume={161},
  pages={1-59},
}

@article{Tits1972FreeSI,
  title={Free subgroups in linear groups},
  author={Jacques Tits},
  journal={Journal of Algebra},
  year={1972},
  volume={20},
  pages={250-270},
}

@article {CL10,
    AUTHOR = {Caprace, Pierre-Emmanuel and Lytchak, Alexander},
     TITLE = {At infinity of finite-dimensional {CAT}(0) spaces},
   JOURNAL = {Math. Ann.},
    VOLUME = {346},
      YEAR = {2010},
    NUMBER = {1},
     PAGES = {1--21},
}

@article {CL,
    AUTHOR = {Caprace, Pierre-Emmanuel and L\'{e}cureux, Jean},
     TITLE = {Combinatorial and group-theoretic compactifications of
              buildings},
   JOURNAL = {Ann. Inst. Fourier (Grenoble)},
    VOLUME = {61},
      YEAR = {2011},
    NUMBER = {2},
     PAGES = {619--672},
}

@article{Bestvina1997TheTA,
  title={The Tits alternative for Out($F_n$) I: Dynamics of exponentially-growing automorphisms},
  author={Mladen Bestvina and Mark Feighn and Michael Handel},
  journal={Annals of Mathematics},
  year={1997},
  volume={151},
  pages={517-623},
}

@article{Genevois,
author = {Anthony Genevois},
title = {{Median Sets of Isometries in $\mathrm{CAT}(0)$ Cube Complexes and Some Applications}},
volume = {71},
journal = {Michigan Mathematical Journal},
number = {3},
pages = {487 -- 532},
year = {2022},
}

@book{nakamura2008some,
  title={Some Results in Topology and Group Theory},
  author={Nakamura, K.},
  year={2008},
  publisher={University of California, Davis}
}

@article{dyubina2000instability,
  title={Instability of the virtual solvability and the property of being virtually torsion-free for quasi-isometric groups},
  author={Dyubina, Anna},
  journal={International Mathematics Research Notices},
  volume={2000},
  number={21},
  pages={1097--1101},
  year={2000},
}

\end{sloppypar}
\end{document}